\documentclass{aims}

\usepackage{txfonts}
\def\typeofarticle{Research Article}
\def\currentvolume{1}
\def\currentissue{x}
\def\currentyear{2018}

\def\ppages{xxx--xxx}
\def\DOI{10.3934/Math.2018.x.xxx}
\def\Received{???}
\def\Accepted{June 2018}
\def\Published{???}

\newtheorem{theorem}{Theorem}

\newtheorem{remark}[theorem]{Remark}

\begin{document}

\title{A renormalization approach to the Riemann zeta function at $-1$, $1+2+3+....\sim-1/12.$}

\author{%
  Gunduz Caginalp\affil{1,}\corrauth
}

\shortauthors{the Author(s)}

\address{%
  \addr{\affilnum{1}}{Mathematics Department, University of Pittsburgh, Pittsburgh, PA 15260, USA}
}

\corraddr{caginalp@pitt.edu; Tel: +1-412-624-8339
}

\begin{abstract} A scaling and renormalization approach to the Riemann zeta
function, $\zeta$, evaluated at $-1$ is presented in two ways. In the first,
one takes the difference between $U_{n}:=\sum_{q=1}^{n}q$ and
$4U_{\left\lfloor \frac{n}{2}\right\rfloor }$ where $\left\lfloor \frac{n}%
{2}\right\rfloor $ \ is the greatest integer function. Using the Cesaro mean
twice, i.e., $\left(  C,2\right)  $, yields convergence to the appropriate
value. For values of $z$ for which the zeta function is represented by a
\textit{convergent} infinite sum, the double Cesaro mean also yields
$\zeta\left(  z\right)  ,$ suggesting that this could be used as an
alternative method for extension from the convergent region of $z.$ In the
second approach, the difference $U_{n}-k^{2}\bar{U}_{n/k}$ between $U_{n}$ and
a particular average, $\bar{U}_{n/k}$, involving terms up to $k<n$ and scaled
by $k^{2}$ is shown to equal exactly $-\frac{1}{12}\left(  1-k^{2}\right)  $
for all $k<n$. This leads to another perspective for interpreting
$\zeta\left(  -1\right)  $.
\end{abstract}


\maketitle

\section{Introduction}

The Riemann zeta function is defined as the analytic continuation of the
infinite sum, $\zeta\left(  z\right)  =\sum_{q=1}^{\infty}q^{-z}$ where
$z=x+iy\in\mathbb{C}$ for $\operatorname{Re}z=x>1.$ For $x>1$ the series
converges absolutely to an analytic function. For all other values of $z$ it
diverges. Riemann showed \cite{t} that it can be continued analytically for
complex values of $z\in\mathbb{C}\ \backslash\left\{  1\right\}  ,$ i.e.,
except for the value corresponding to the harmonic series. For $z=-1$ one has
the (divergent) sum of natural numbers. The analytic continuation of the
series yields the result $\zeta\left(  -1\right)  =-1/12,$ with the formal
representation that appears to be an obvious contradiction:%
\begin{equation}
1+2+3+\ ...\ \sim\zeta\left(  -1\right)  =-\frac{1}{12}\ . \label{sum}%
\end{equation}
In this note we examine this relation using an approach that involves scaling
the truncated (finite) sum and renormalizing in order to obtain a finite
result as one takes the infinite limit of the sum.

Renormalization consists of a set of methodologies constituting a philosophy
and approach to problems exhibiting a divergence in some form. Originally
introduced for statistical mechanics and quantum field theory by Ken Wilson in
the 1970's, renormalization was able to yield the exponents with which key
physical properties diverge (see for example, \cite{w}, \cite{cfp}). The basic
idea is first to average spins within a particular geometric configuration,
thereby reducing the size of the system by a factor greater than unity. The
reduction in size must be compensated by adjusting the interaction strengths.
If this were not done, then iteration of this process would yield a trivial
fixed point of zero or infinity. With the appropriate renormalization,
however, one can iterate the procedure repeatedly. The key ansatz is that the
exponent of the divergent quantity should not change due to this averaging
process\ (with the interactions appropriately renormalized) since the
singularity is due to the divergence of the "correlation length" which is the
a measure of the distance at which spins can influence one another.

This approach to statistical mechanics revolutionized many calculations, as
very simple calculations yielded the results previously obtained by a
\textit{tour de force}, and led to its adaptation in a number of other areas.
The text by Creswick,\ Poole and Farach \cite{cfp} describes the
implementation of this approach to classical mathematical problems such as
fractals and random walk. For example, in random walk, the classical result
under robust conditions is that after $n$ steps the random walk has mean
distance $n^{1/2}$ from the original point. This result can also be obtained
by averaging sets of $k$ steps, and readjusting (i.e., renormalizing) the step
size so that one considers a walk of $n/k$ steps with the new step size. The
unique renormalization (i.e., setting the new step size) that leads to a
non-trivial result (i.e., not $0$ or $\infty$) yields the exponent $1/2$ in
$n^{1/2}.$

In this paper we describe methodology along the lines of this approach to
obtain an analog of $\left(  \ref{sum}\right)  $ that is well-defined.

The expression $\left(  \ref{sum}\right)  $ has been of interest in
applications such as string theory \cite{p}. In addition to this perspective,
two physicists \cite{pc} have also provided an explanation of $\left(
\ref{sum}\right)  $ based on shifting infinite sums.

\section{Averaging and re-scaling (Method 1)}
Using the notation $U_{n}=\sum_{q=1}^{n}q$ and $\left\lfloor r\right\rfloor $
as the greatest integer less than or equal to $r$ we define
\[
Y_{n}=U_{n}-4U_{\left\lfloor \frac{n}{2}\right\rfloor }%
\]
One has from a simple calculation,%
\[
Y_{n}=\left\{
\begin{array}
[c]{ccc}%
-n/2 & if & n\ even\\
\left(  n+1\right)  /2 & if & n\ odd
\end{array}
\right.  \ .
\]

Now, let $Z_{N}$ be the Cesaro mean of $\left\{  Y_{n}\right\}  $, which is
also known as the Cesaro sum and plays an important role in Fourier analysis
(see for example, \cite{s}, p.52), i.e.,%

\[
Z_{N}:=Avg\left\{  Y_{n}:n\leq N\right\}  =\frac{1}{N}\sum_{n=1}^{N}Y_{n}.
\]
Considering the odd and even terms separately, one can readily observe that
\begin{align}
Z_{2K+1}  &  =\frac{1}{2K+1}\left\{  1-1+2-2+...+\frac{2K+2}{2}\right\}
\nonumber\\
&  =\frac{1}{2}+\frac{1}{4K+2}\ .\nonumber\\
Z_{2K}  &  =0. \label{z}%
\end{align}
Let $X_{M}:=Avg\left\{  Z_{N}:N\leq M\right\}  $, i.e., the Cesaro mean of
$Z_{N},$ which is the Cesaro mean of the Cesaro mean, i.e., $\left(
C,2\right)  ,$ of the original $\left\{  Y_{n}\right\}  $.

\begin{theorem}
For $K\in\mathbb{N}$ one has the following:%

\begin{equation}
X_{2K}=\frac{1}{4}+\frac{1}{8K}\sum_{j=1}^{K}\frac{1}{j-1/2},\ \ \ \ X_{2K+1}%
=\frac{K+1}{4K+2}+\frac{1}{8K+4}\sum_{j=1}^{K+1}\frac{1}{j-1/2} \label{ident}%
\end{equation}%
\begin{equation}
\left\vert X_{2K}-\frac{1}{4}\right\vert \leq\frac{1}{3K}+\frac{\log\left(
K+1\right)  }{8K} \label{even}%
\end{equation}

\begin{equation}
\left\vert X_{2K+1}-\frac{1}{4}\right\vert \leq\frac{16/3+\log\left(
K+1\right)  }{8K+4} \label{odd}%
\end{equation}
and thus the limit
\begin{equation}
\lim_{M\rightarrow\infty}X_{M}=\frac{1}{4}, \label{lim}%
\end{equation}
which can also be expressed as%
\begin{equation}
\lim_{M\rightarrow\infty}Avg_{N\leq M}\left\{  Avg_{n\leq N}\left\{
\frac{U_{n}-4U_{\left\lfloor \frac{n}{2}\right\rfloor }}{1-4}\right\}
\right\}  =-\frac{1}{12}\ .
\end{equation}
\end{theorem}

\begin{proof}
A computation using $\left(  \ref{z}\right)  $ for even and
odd values of $M$ results in $\left(  \ref{even}\right)  $ and odd $\left(
\ref{odd}\right)  .$ The inequalities$\ $%
\[
\log\left(  K+1\right)  \leq\sum_{j=1}^{K+1}\frac{1}{j-1/2}=\frac{8}{3}%
+\sum_{j=3}^{K+1}\frac{1}{j-1/2}\leq\frac{8}{3}+\log\left(  K+1\right)
\]
yield the result $\left(  \ref{lim}\right)  $.
\end{proof}

\begin{remark}
One can summarize this heuristically as
\begin{align}
E\left[  U_{n}-4U_{\lfloor\frac{n}{2}\rfloor}\right]   &  =P\left\{
n=odd\right\}  \left(  \frac{n+1}{2}\right)  +P\left\{  n=even\right\}
\left(  -\frac{n}{2}\right) \nonumber\\
&  =\frac{1}{2}\left(  \frac{n+1}{2}-\frac{n}{2}\right)  =\frac{1}{4}.
\label{k=2}%
\end{align}
\end{remark}

In order to make $\left(  \ref{k=2}\right)  $ precise, one would need to
invoke some basic probabilistic ideas, primarily the Kolmogorov extension, or
existence, theorem (see for example, \cite{b}, p. 514) whereby the probability
on a finite subset of $\mathbb{N}$ can be extended to all of $\mathbb{N}$.
Formally identifying $U_{n}$ and $U_{\lfloor\frac{n}{2}\rfloor}$ as
$n\rightarrow\infty$ as though they were convergent leads to $\left(
\ref{sum}\right)  $.

\bigskip

One can define the analogous relations $U_{n}\left(  z\right)  :=\sum
_{q=1}^{n}q^{-z}$ which, as noted above, converges to $\zeta\left(  z\right)
$ for $\operatorname{Re}z>1$. Note that if a series that converges in the
ordinary sense, the Cesaro mean must converge to the same limit. Moreover, if
a sequence $\left\{  c_{j}\right\}  $ is convergent to $c$, the Cesaro mean
also converges to $c$. Thus, for values of $z\in\mathbb{C}$ in the convergent
region, the infinite sum is equal to the Cesaro mean, so that the Cesaro mean
can be used for both convergent and nonconvergent values.

\bigskip

In particular for a value $z$ for which $U_{n}\left(  z\right)  $ converges in
the usual sense to $\zeta\left(  z\right)  $, one has%
\[
\lim_{n\rightarrow\infty}Y_{n}\left(  z\right)  =\lim_{n\rightarrow\infty
}\left\{  U_{n}\left(  z\right)  -4U_{\left\lfloor \frac{n}{2}\right\rfloor
}\left(  z\right)  \right\}  =-3\zeta\left(  z\right)
\]
as $n\rightarrow\infty$ . Since $Y_{n}\left(  z\right)  $ is convergent for
this value of $z,$ it follows that the Cesaro mean $Z_{N}\left(  z\right)
:=N^{-1}\sum_{n=1}^{N}Y_{n}\left(  x\right)  $ also converges to the same
limit, $-3\zeta\left(  z\right)  $. Similarly, the Cesaro mean, $X_{M}\left(
z\right)  :=M^{-1}\sum_{N=1}^{M}Z_{N}\left(  z\right)  $ also converges to
$-3\zeta\left(  z\right)  $. Thus, the interpretation that double Cesaro mean
of
\[
\frac{U_{n}\left(  z\right)  -4U_{\left\lfloor \frac{n}{2}\right\rfloor
}\left(  z\right)  }{1-4}%
\]
converges to $\zeta\left(  z\right)  $ is maintained for values of $z$ for
which $U_{n}\left(  z\right)  =\sum_{q=1}^{n}q^{-z}$ is convergent.

For example, setting $z=2,$ so that one has a convergent series,%
\[
\zeta\left(  2\right)  =\sum_{n=1}^{\infty}n^{-2}=\frac{\pi^{2}}%
{6},\ i.e.,\ \ \lim_{n\rightarrow\infty}Y_{n}\left(  2\right)  =\left(
-3\right)  \frac{\pi^{2}}{6},
\]
and consequently $\lim_{M\rightarrow\infty}X_{M}\left(  z\right)  /\left(
-3\right)  =$ $\pi^{2}/6.$

Hence, this approach using the double Cesaro mean may present another avenue
to extend the sum from the convergent to the nonconvergent regions, and offer
other ways to study the Riemann zeta function.

\section{Averaging and re-scaling (Method 2)}

An alternate approach to averaging (without using the greatest integer
concept) can be implemented by fixing $k$ and letting $n=mk+j.$

Note that for $r\in\mathbb{N}$ one has $U_{r}=\sum_{q=1}^{r}q=\frac{1}%
{2}r\left(  r+1\right)  .$ Define a continuous extension of $U_{r}$ by
$U_{r}:=\frac{1}{2}r\left(  r+1\right)  $ to $r\in\mathbb{R}$. \ Then define
the average of over the values of $j\in\left\{  0,...,k-1\right\}  $ as%
\begin{equation}
\bar{U}_{\frac{n}{k}}:=\frac{1}{k}\sum_{j=0}^{k-1}U_{\frac{n-j}{k}}.
\label{avg}%
\end{equation}

\begin{theorem}
For $n,k\in\mathbb{N}$ and $k<n$ one has the exact
relation $\ $%
\begin{equation}
\frac{U_{n}-k^{2}\bar{U}_{\frac{n}{k}}}{1-k^{2}}=-\frac{1}{12}. \label{k}%
\end{equation}
\end{theorem}

\begin{proof}
From the definition $\left(  \ref{avg}\right)  $ one observes
\begin{align*}
U_{n}-k^{2}\bar{U}_{\frac{n}{k}}  &  =\frac{1}{2}n\left(  n+1\right)
-k^{2}\frac{1}{k}\sum_{j=0}^{k-1}\frac{1}{2}\left(  \frac{n-j}{k}\right)
\left(  \frac{n-j}{k}+1\right) \\
&  =-\frac{1}{12}\left(  1-k^{2}\right)  .
\end{align*}

Thus, the quotient in $\left(  \ref{k}\right)  $ is thus independent of both
$k$ and $n.$ This averaging and scaling of $U_{n}$ by any factor $k$ together
with renormalization by $k^{2}$ yields the unique number $-1/12.$
\end{proof}

Formally, the identification of $U_{n}$ and $\bar{U}_{n/k}$ with $U$ in the
limit as $n\rightarrow\infty$ leads to $\sum_{q=1}^{\infty}q\sim
\zeta\left(  -1\right)  =-1/12.$

\bigskip

\begin{remark}
(a) The number $k$ can be regarded as the analog of the
"subwalk" of $k$ steps\ in the random walk problem discussed in the
introduction, with $n/k$ the new walk with the mean step size increased by a
factor $k^{2}.$

\bigskip

(b) Setting $k:=n^{p}$ with $p\in\left(  0,1\right)  $ one can write expression
$\left(  \ref{k}\right)  $ in the form%
\[
\frac{n^{-p}U_{n}-n^{p}\bar{U}_{n^{1-p}}}{n^{-p}-n^{p}}=-\frac{1}{12}.
\]

\bigskip

(c) Applying this approach for $k=2$ yields,%

\begin{align*}
\bar{U}_{\frac{n}{2}}  &  =\frac{1}{2}\left(  \frac{\frac{n}{2}\left(
\frac{n}{2}+1\right)  }{2}+\frac{\frac{n-1}{2}\left(  \frac{n-1}{2}+1\right)
}{2}\right)  =\allowbreak\frac{1}{8}n+\frac{1}{8}n^{2}-\frac{1}{16}\\
U_{n}-4\bar{U}_{\frac{n}{2}}  &  =\allowbreak\frac{1}{4}%
\end{align*}
so that formally identifying $U_{n}$ and $\bar{U}_{\frac{n}{2}}$ as $U$ in the
limit $n\rightarrow\infty$ yields $U\sim-\frac{1}{12}.$

\bigskip

(d) To extend this result to other values of $z\in\mathbb{C}$,\ one can utilize
again the analogous quantity, $U_{n}^{\left(  z\right)  }=\sum_{q=1}^{n}%
q^{-z}$, determine whether it is possible to formulate a definition analogous
to $\left(  \ref{avg}\right)  ,$ and consider values of $z $ for which these
converges. The left hand side of $\left(  \ref{k}\right)  $ can be considered
in the same manner as described for Method 1.

\end{remark}

\section*{Acknowledgments}
The author thanks Dr. Alban Deniz and Prof. Bogdan
Ion for useful discussions.

\section*{Conflict of Interest} The author declares no conflict of interest.

\end{document}